\documentclass[a4paper]{article}
\usepackage[margin=1in]{geometry}
\usepackage{setspace}
\usepackage{setspace}
\usepackage{graphicx}
\usepackage{enumitem}
\usepackage{amsthm}
\usepackage{amssymb,amsmath,amsthm, amsfonts}
\usepackage{tipa}
\usepackage{dsfont}
\usepackage{amsmath}
\usepackage[english]{babel}

\usepackage{epstopdf}

\usepackage[english]{babel}
\usepackage[utf8x]{inputenc}
\usepackage[T1]{fontenc}

\usepackage{amsmath}
\usepackage{graphicx}
\usepackage[colorinlistoftodos]{todonotes}
\usepackage[colorlinks=true, allcolors=blue]{hyperref}

\title{Introduction to the $p$-adic Space}
\author{Joel P. Abraham}

\begin{document}
\maketitle

\section{Introduction}

From a young age, students learn fundamental operations with the real numbers: addition, subtraction, multiplication, division, etc. We intuitively understand that distance under the reals as the absolute value function, even though we never have had a fundamental introduction to set theory or number theory. As such, I found this topic to be extremely captivating, since it almost entirely reverses the intuitive notion of distance in the reals. Such a small modification to the definition of distance gives rise to a completely different set of numbers under which the typical axioms are false. Furthermore, this is particularly interesting since the $p$-adic system commonly appears in natural phenomena. Due to the versatility of this metric space, and its interesting uniqueness, I find the $p$-adic space a truly fascinating development in algebraic number theory, and thus chose to explore it further.

\section{The $p$-adic Metric Space}
\subsection{$p$-adic Norm}
\theoremstyle{definition}
\newtheorem{definition}{Definition}[section]
\newtheorem{theorem}{Theorem}[section]
\newtheorem{corollary}{Corollary}[theorem]

\begin{definition}{}
A \textit{metric space} is a set with a distance function or \textit{metric}, $d(p, q)$ defined over the elements of the set and mapping them to a real number, satisfying:
\begin{description}[style=unboxed,leftmargin=2cm]
  \item[(a) $d(p,q)>0$ if $p\neq q$]
  \item[(b) $d(p,p)=0$]
  \item[(c) $d(p,q)=d(q,p)$]
  \item[(d) $d(p,q) \leq d(p,r) + d(r, q)$]
\end{description}
\end{definition}

The real numbers clearly form a metric space with the absolute value function as the norm such that for $p, q \in \mathbb{R}$ 
  \[d(p, q) = |p-q|.\]
 
We can modify the distance function to alter the notion of distance within a space, thus redefining closeness. In the reals, two numbers are close together if their decimal representation is similar towards the right. For example, $3.1, 3.14, 3.141, 3.1415,...$ are all growing closer to each other, while $15, 315, 1415, 31415,...$ are all growing further apart. This is simply because appending a number $n$ digits to the right of the decimal place only modifies its distance from another real number by $10^{-n}$, while appending a number $n$ digits to the left of the decimal place alters its distance from another real by $10^{n}$. The $p$-adic numbers are quite the opposite. The $p$-adic space redefines the distance function, informally, such that two numbers are far if they are similar towards the right and close if they are similar towards the left. This is made rigorous through the $p$-adic norm. 

\begin{definition}{}
We can represent any rational number $n$ as $n=\frac{p^nm}{n}$ where $p$, $m$, and $n$ are coprime. Then, the \textit{$p$-adic valuation} $v_p^\mathbb{Z}(n)=p^n$. Formally, the \textit{$p$-adic valuation} over the integers is a function $v_p^\mathbb{Z}(n): \mathbb{Z} \rightarrow \mathbb{N}$ such that  
\[v_p^\mathbb{Z}(n):= \begin{cases} 
      +\infty & n= 0 \\
     \max{\left\{\,\,v \,:\,\, p^v \,|\, n, \,v \in \mathbb{N} \right\}} & n \neq 0
   \end{cases}
\]

The $p$-adic valuation can be extended to the rationals as well, and the function $v_p(n): \mathbb{Q} \rightarrow \mathbb{Z}$ is defined as $v_p(\frac{a}{b})= v_p^\mathbb{Z}(a) - v_p^\mathbb{Z}(b)$.
\end{definition}

\begin{definition}{}
The $p$-adic valuation is an auxillary function. It is used to clarify the definition of the \textit{$p$-adic norm}, which is defined as $|\cdot|_p:\mathbb{Q}\rightarrow \mathbb{N} $ such that

\[|n|_p:= \begin{cases} 
      0 & n= 0 \\
     p^{-^{v_p(n)}} & n \neq 0
   \end{cases}
\]

\end{definition}

By establishing the definition of the $p$-adic norm, we are now able to form a metric space over the $p$-adic numbers with the distance function defined as $d(p, q)= |p-q|_p$. We can now make some observations about the nature of this space. The size of a $p$-adic number can be informally summarized as inversely proportional to the exponent of p in its prime factorization. It is clear that numbers that are larger in the real number space, per the absolute value function, tend to be  smaller under the $p$-adic metric, since these numbers will be divisible by greater powers of $p$. Note that any number that is coprime to $p$ will have size 1, since the valuation returns 0.

\subsection{Completion of $\mathbb{Q}$}

The set of rational numbers is the fundamental structure underlying much of number theory and mathematics as a whole.  

\[\mathbb{Q} = \left\{\frac{p}{q} \,\,| \,\, p,q \in \mathbb{Z},\, q \neq 0 \right\}.\]

The rationals form a field under the  addition and multiplication, satisfying the field axioms for each operation: commutativity, associativity, existence of an identity element, existence of an inverse element, and closure. However, upon examination we quickly discover inefficiencies in this field; for example, there does not exist a rational $m$ such that $m^2=2$. An irrational number, $\sqrt{2}$ is one of the many gaps in the rational field, since we can get arbitrarily close to it from within the rationals but never reach it exactly. Closing these gaps under the real distance function yields the real numbers, while closing these gaps using the $p$-adic distance function yields the $p$-adic numbers. Before we can complete the rations, however, we must introduce some important definitions relating to elementary topology.

\begin{definition}{}
For a metric space X, we define an \textit{open ball} to be an open set $B(x_0, r)$ of all points $x \in X$ centered around a point $x_0$ with \textit{radius} $r$ such that $d(x_0, x)<r$. The corresponding \textit{closed ball} is said to be the closure of an open ball.
\end{definition}

\begin{definition}{}
A \textit{neighborhood} of $x_0$ is said to be an open ball $N_r(x_0)$ centered around a point $x_0$ with \textit{radius} $r$.

\end{definition}

\begin{definition}{}
A \textit{limit point} of a set $X$ is point $p$ if every neighborhood of $p$ overlaps with $X$, that is, for every neighborhood of $p$, there exists a point $q\neq p$ such that $q \in X \cap N_r(p)$.
\end{definition}

\begin{definition}{}
A \textit{Cauchy sequence} is any sequence ${x_i}$ such that for every positive real number $\epsilon$, there exists a natural number $m$ such that $d(x_m, x_{m+n})< \epsilon$ for any $n \in \mathbb{N}$.
\end{definition}

Consider the sequence $X=\{1, 1,6, 1.61, 1.614,...\}$ in the real number space with each term including a subsequence digit of the decimal representation of $\sqrt{2}$. We can show that this is a Cauchy sequence by noting that all terms in the subsequence of $X$ beginning with $x_i$ are within $10^{1-i}$ of each other. Thus, for any $\epsilon$, we can choose $i$ such that $1-\log{\epsilon} \leq i$, which would give us a sequence $\{x_i, x_{i+1},...\}$ such that for any $n \in \mathbb{N}$, $d(x_i,\, x_{i+n}) \leq 10^{1-i} \leq \epsilon$. A limit point of this Cauchy sequence is therefore $\sqrt{2}$ since for any $r>0$, we can choose an $\alpha > 0 $ and construct a neighborhood centered around $\sqrt{2}$ such that $d(\sqrt{2}, \sqrt{2}-r-\alpha) = |r-\alpha| \in X$. This is because of the fact for the $n^{th}$ term in the sequence is within $|\sqrt{2}- x_n| \leq 10^{1-n}$. In the rational space, all the limit points such as $\sqrt{2}$ that are formed under the absolute value function are contained in the set of irrational numbers. As such, we may construct the real numbers by completing the rationals, that is, including all the limit points of $\mathbb{Q}$ to yield $\mathbb{R}$. This proof is quite complex and thus will not be covered in this paper, however, the general structure of the argument behind this proof is important since the $p$-adic numbers are constructed in the same fashion. 

While the real numbers represent the completion of $\mathbb{Q}$ under the the absolute value function which is the distance in the reals, the completion of $\mathbb{Q}$ induced by the $p$-adic metric forms the $p$-adic number set, denoted by $\mathbb{Q}_{p}$. If we only consider integral numbers, we arrive at the set of $p$-adic integers, a subring of $\mathbb{Q}_p$ denoted by $\mathbb{Z}_p$.

The formal representation of a $p$-adic number is very similar to that of a $p$-ary number (a number in base $p$), as both are represented by a power series expansion of $p$. 

\begin{definition}{}
A \textit{$p$-adic number} $\alpha$ can be represented, uniquely, for $ 0 \leq a_i < p$ as
\[\alpha = \sum_{i=n}^{\infty}a_i p^{i}\]
where the $p$-adic norm of $\alpha$, $|\alpha|=p^{-n}$.

\end{definition}

\subsection{Arithmetic Operations}

One interesting consequence of the $p$-adic metric is that numbers no longer require a sign. That is, negative numbers are defined without the need for a negative sign. To illustrate this property, consider the $10$-adic metric. In this space, considering a series of $x$ $9$s for increasing values of $x$, our number approaches $-1$. This is because 

\[|9-(-1)|_{10} =|9| = 10^{-1}, \,\,\,|99-(-1)|_{10} = |10^2| = 10^{-2},\,\,\,...\,\,\,|10^{x}-1-(-1)|=|10^x|=10^{-x}\] \[\lim_{x\to\infty}|10^{-x}|_{10} = 0\].

Thus, the distance between $999...9$ and $-1$ is considered to be $0$. Since this technique involves constantly appending to the left, in the 10-adic system, we can represent this same process involving $\lim_{x\to\infty}|10^{x}-1|_{10}$ as $...999=-1$.

This ring of $10$-adic (decadic) numbers is very useful to demonstrate the relationship between the real numbers and the $p$-adic numbers. Let us now explore the function of transformations via the fundamental operations in $10$-adic space. We can define the  basic operations in the $10$-adic space: addition and multiplication, thus they form a commutative ring (note that while $p$-adic spaces form fields for some values of $p$, all $p$-adic spaces form rings). Addition and multiplication with $p$-adic numbers function the same as with real numbers, the only difference being that $p$-adic numbers can extend infinitely to the left. To demonstrate addition, consider $...111123 + 7 = ...111130 $ and $...111123 - 7 = ...111116 $. Furthermore, the existence of unsigned negatives can be demonstrated through $....999942 + 58 = 0$; we know that $...999942$ and $58$ are additive inverses, thus we consider $...999942 = -58$. Multiplication is also defined ordinarily, with $...111123 \cdot 7 = ...777961 $. We also have $...6667 \cdot 3 = ...0001$, which implies that $...6667$ and $3$ are multiplicative inverses, meaning that we consider $...6667 = \frac{1}{3}$. In this instance, division is defined, however, we cannot define division for any two $p$-adic numbers, since it is possible to have  two numbers $x, y$ for which $x \cdot y = 0$.

\subsection{Properties of $p$-adic Space}
We will now explore various properties of the set of $p$ adic numbers that emerge from our redefinition of the distance function. This section introduces some fundamental theorems in $p$-adic analysis that are incredibly useful for manipulating numbers within $p$-adic space.

\begin{theorem}

The $p$-adic norm satisfies the \textit{strong triangle inequality}:
\[|x+y| \leq max{\{|x|, |y|\}}\]

\end{theorem}

\begin{proof}
    For $x=0$ or $x=y$ we have $|y|_p = |y|_p$, thus we need to prove the statement true for $x, y \neq 0$ and $x \neq y$. Without loss of generality, assume $|x|_p>|y|_p$. Since $x=\frac{p^{a}q}{r}$ and $x=\frac{p^{b}s}{t}$ for some $a, b, p, q, s, t \in \mathbb{N}$ with coprime $p, q, r$ and $p, s, t$, implies $|x|_p=p^{-a}$  and $|y|=p^{-b}$, we must have $a\leq b$. Thus, we need to show that $|x+y|_p \leq p^{-a}$. We can write this as
    
    \[|\frac{p^{a}q}{r}+\frac{p^{b}s}{t}| \leq p^{-a} \,\,\implies\,\, |p^{a}(\frac{q}{r}+\frac{(p^{b-a})s}{t})| \leq p^{-a}\].
    
    We observe that the $p$-adic valuation of a number $n$ is the exponent of $p$ in the prime factorization of $n$; let $A=\frac{q}{r}+\frac{(p^{b-a})s}{t}$, then we have
    
    \[|p^{a}A| = p^{-v_{p}(p^{a}A)} = p^{-a \cdot v_{p}(A)} = p^{-a}\cdot p^{-v_{p}(A)}\].
    
    Since the image $v_{p}(n)$ is the natural numbers, $-v_{p}(n) \leq 0$, thus $p^{-v_{p}(A)} \leq 1$ so 
   \[ |x+y| = p^{-a}\cdot p^{-v_{p}(A)} \leq p^{-a}\].
\end{proof}

\begin{corollary}

The strong triangle inequality clearly implies the normal \textit{triangle inequality}, that is
\[|x+y| \leq |x|+|y|\]

\end{corollary}

\begin{proof}
    This is trivial with Theorem 2.1, since we have
    \[|x+y| \leq max{\{|x|, |y|\}} \leq |x|+|y|\]
\end{proof}

\begin{theorem}[\textbf{Product Formula}]

For $\alpha \in \mathbb{Q}$, $\alpha \neq 0$, let $v$ represent the set of prime numbers and infinity $v=\{\infty, 2, 3, 5, 7,...\}$ and define the $\infty$-adic norm to be the real absolute value function such that $|\alpha|_{\infty}=\alpha$ \cite{gouvea}. Then,
\[\prod_{v} |\alpha|_{v} =1\]

\end{theorem}

\begin{proof}
    We prove this by considering three cases corresponding to the three posssible categories into which we can classify $v$: $v = \infty$, $v | \alpha$, $v \nmid \alpha$. First, represent $\alpha$ by its prime factorization, namely, $\alpha=p_1^{n_1} \cdot p_2^{n_2} \cdot p_3^{n_3}\cdot...\,p_i^{n_i}$. Then we consider the cases:
    \[\prod_{v} |\alpha|_{v} = \begin{cases} 
      1 & v \neq n_i \\
     p_1^{-n_1} \cdot p_2^{-n_2}\cdot...\,p_i^{-n_i} = \frac{1}{\alpha} & v = n_i\\
     \alpha & v=\infty
   \end{cases}\]
    
    Since these three cases encompass all possible values of $v$, we can compute the value of the infinite product by taking the product of these three subproducts, yielding
    \[\prod_{v} |\alpha|_{v} = 1\cdot \frac{1}{\alpha} \cdot \alpha =1\].
\end{proof}

This theorem is a useful result and in fact, one of the reasons why the $p$-adic metric was defined as it is. The Product Formula is quite powerful and can be used to solve a wide array of problems in number theory.

\begin{theorem}[\textbf{Ostrowski's Theorem}]

Every nontrivial norm (where a trivial norm is defined as yielding $0$ if $x=0$ or $1$ if $x \neq 0$) over the rational numbers is equal to either the real norm or the $p$-adic norm

\end{theorem}

This theorem essentially implies that the only two possible completions of the rations are the reals and the $p$-adic numbers. Recall that the rationals are completed by including the limit points of all Cauchy sequences in the rationals, and note that the only difference between the real completion and the $p$-adic completion is the differing distance function. Such a discovery has profound implications for algebraic number theory and thus represents a deep result. Ostrowski's Theorem implies that the reals and $p$-adic numbers are each powerful rings, thus mathematicians find it useful to synthesize these into a single structure. This simplifies manipulation of a number in $p$-adic space and lets us examine the influence of modifying the value of $p$.

\begin{definition}{}
The \textit{ring of adeles} is the set $\mathbb{A}$, given a certain number $\alpha$, in $p$-adic space for all $p$ with

\[\mathbb{A}  = \{\alpha_{\infty}, \alpha_2, \alpha_3, \alpha_5,...\}\].

\end{definition}

\section{Applications of $p$-adic Numbers}

While the $p$-adic space is undoubtedly fascinating, one cannot help but interrogate its pragmatic value. This section will survey the various applications of $p$-adic theory (outside of number theory and pure mathematics) in physics and computer science, and briefly detail the value of $p$-adic numbers. Somewhat counterintuitively, the $p$-adic space emerges consistently in natural phenomena.

\subsection{Quantum Mechanics}

There are a few import distinctions between the real numbers and the $p$-adic numbers that become import when we consider their respective axiomatic systems as a framework for modeling the universe, that is, when we apply them to physics. \cite{rozikov} One key distinction is that we say the real numbers are Archimedean, since they satisfy the Archimedean property: informally, given any real numbers $x$ and $y$, we can add $x$ to itself enough times until it exceeds $y$. Formally, $nx > y$ for $n \in \mathbb{N}$ and $x, y \in \mathbb{R}$. However, the $p$-adic numbers are so-called non-Archimedean, since they satisfy the ultra-metric triangle inequality, or strong triangle inequality, and therefore does not satisfy the Archimedean property. This can be trivially shown.

\begin{theorem}

The $p$-adic numbers are non-Archimedean, that is, for any $x, y \in \mathbb{R}$ we cannot choose $n \in \mathbb{N}$ such that
\[nx > y\].

\end{theorem}

This is a corollary of the strong triangle inequality, since, informally, if we set $x=y$, \[|x+y| \leq |x|\,\, \implies \,\,|2x| \leq |x|,\] which hints at the fact that adding $x$ to itself decreases its value under the $p$-adic metric. 

This is useful in quantum mechanics, since it allows us to represent measurements that would ordinarily be small, within real spacetime, as large values in $p$-adic spacetime. As such, $p$-adic numbers have tremendous value for physicists attempting to develop a mathematical framework for spacetime in distances near the Planck length, where our typical standards of measurement fail. 

\subsection{Computer Science}
One excellent example of the utility of $p$-adic numbers is its enabling of infinite precision arithmetic. Traditional methods of computing use the real numbers and include a margin of error. Numbers are represented via floating points, and therefore manipulation is considered to be finite precision arithmetic due to roundoff error. For example, computers represent $\pi$ to a certain number of digits, thus operations with $\pi$ are never truly precise. However, representing numbers in the $p$-adic space, we can transform some numbers that would extend infinitely in its real decimal representation into a number of finite length. This feature can be exploited for error-free arithmetic. Furthermore, $p$-adic numbers are useful for algorithms for fast integer multiplication as demonstrated by \cite{de}.

The field of $p$-adic numbers has no order structure, which is a surprisingly useful property in cryptography and cryptanalysis. While an order is defined on the real numbers, we cannot define a linear ordering over the set $\mathbb{Q}_p$. This  can be summarized as a consequence of the proof that $\mathbb{Q}_p$ must contain transcendental (non-algebraic) elements, since while $\mathbb{Q}_p$ is uncountable, the algebraic elements of $\mathbb{Q}_p$ are countable, implying the existence of non-algebraic elements. This lack of structure is valuable in the context of public key cryptography -- the more a certain transformation obfuscates a structure, the more encrypted it is, and thus the more difficult it is to invert the operation. The chaos inherent in the $p$-adic numbers is a significant benefit in the context of cryptography, and much research at the intersection of math and computer science is devoted to exploring cryptography on elliptic curves over $p$-adic number fields. In particular, many papers introduce cryptosystems defined over the quotient groups of an elliptic curve's rational points, over a $p$-adic field.

\section{Reflection}

Given the fascinating, numerous applications of $p$-adic numbers, they represent a profound development in algebraic number theory. I discussed the applications of $p$-adic numbers in the context of physics and computer science, but there are many more uses of $p$-adic numbers in the realm of number theory, most notably in the proof of Fermat's Last Theorem. 

Throughout this investigation, I have learned a lot about elementary topology and the various ways in which distance is used to construct number sets. I also gained insight into proof structures and logical continuity. Before this investigation, I was not familiar at all with the rigorous definitions of the absolute value function, the real numbers, the rational numbers, metric spaces, and the $p$-adic space. Due to the thorough research I conducted in this area and the depths of my investigation, I have realized profound things about the nature of reality, which are fundamentally interdisciplinary and can be applied to literature, philosophy, history, etc. From this investigation, I realized that the framework of math is, at least somewhat, arbitrarily defined and thus inherently dynamic, that we can redefine everything and thus alter the structure of mathematics at its core.

\bibliographystyle{alpha}

\end{document}